\documentclass[12pt]{amsart}
\usepackage{amsmath,amssymb,amsthm}
\usepackage{color}
\usepackage{hyperref}

\newtheorem{lem}{Lemma}[section]

\newtheorem{thm}{Theorem}[section]

\theoremstyle{definition}

\theoremstyle{remark}

\theoremstyle{remark}
\newtheorem{remark}{Remark}[section]
\numberwithin{equation}{section}

\newcommand{\C}{{\mathbb C}}

\newcommand{\R}{{\mathbb R}}

\definecolor{blu}{rgb}{0,0,1}

\title{An alternative variational approach to constrained critical points
for nonlinear Schr\"odinger  equations with combined nonlinearities}

\begin{document}

\author{Jacopo Bellazzini}
   \address{Jacopo Bellazzini \newline Dipartimento di Matematica  \\ Universit\`a di Pisa \\ Largo B. Pontecorvo 5, 56100 Pisa, Italy}%
\author{Pietro Majer}
\address{Pietro Majer
\newline Dipartimento di Matematica Universit\`a di Pisa
Largo B. Pontecorvo 5, 56100 Pisa, Italy}%

\begin{abstract} We introduce a new variational framework for the study of standing waves for  nonlinear Schr\"oedinger equations with combined power nonlinearities. The approach is based on minimizing the mass functional over a set defined by an upper bound on the energy and a lower bound on a nonlinear potential term, rather than minimizing the energy at prescribed mass.  This yields a new characterization of the sharp Gagliardo--Nirenberg optimizer through purely variational constraints. As a consequence, we establish the existence of minimizers for all energy levels below a critical threshold and derive the corresponding Euler--Lagrange equations. These minimizers give rise to standing-wave solutions of the nonlinear Schr\"oedinger  equation.\end{abstract}

\maketitle
In this paper we are motivated by  the following  Schr\"odinger type equation
\begin{equation}\label{SP}
i\psi_{t}+ \Delta \psi - |\psi|^{q-1} \psi+|\psi|^{p-1}\psi=0,
\end{equation} 
where $\psi(t,x):[0,T)\times\R^{d}\rightarrow \C$, $d\geq 3$,  is a complex wave function
 and the exponents $q, p$ are of mass subcritical type, meaning that and $1<q< p<1+\frac{4}{d}$.
 Given an initial datum $\psi_0\in H^1(\R^d)$, the Cauchy problem is (globally) well-posed and satisfies conservation of mass and energy, defined respectively by 
\begin{equation}
    M(\psi(t))=\int_{\R^d}|\psi(t,x)|^2dx
\end{equation}
and
\begin{equation}
    E(\psi(t))=\frac12\int_{\R^d}|\nabla \psi(t,x)|^2dx+\frac{1}{q+1}\int_{\R^d}| \psi(t,x)|^{q+1}dx-\frac{1}{p+1}\int_{\R^d}| \psi(t,x)|^{p+1}dx.
\end{equation} 
Particular solutions to \eqref{SP}
are the so-called standing wave solutions, which are   solutions to \eqref{SP} of the form
 \[
 \psi(t,x)=e^{i\omega t}u(x)
 \]
 with $\omega\in \R$, and $u(x)\in \C$ is a time-independent function belonging to  $H^1(\R^d)$ which satisfies 
\begin{equation}\label{eq}
- \Delta u + \omega u +|u|^{q-1} u-|u|^{p-1}u=0.
\end{equation}
As the mass is a physical quantity which is preserved along the flow of \eqref{SP}, in order to find solutions $u$ to \eqref{eq}, a natural way to study the problem is to look for  critical points of the energy functional 
 \[
 E(u)=\frac{1}{2}\int_{\R^{d}}|\nabla u|^{2}dx+\frac{1}{q+1}\int_{\R^{d}}| u|^{q+1}dx-\frac{1}{p+1}\int_{\R^{d}}|u|^{p+1}dx
 \]
constrained on the $L^{2}$-spheres in $H^{1}(\R^{d};\C)$ described by 
\[S_{\rho}=\{u\in H^{1}(\R^{d};\C) \ : \  \|u\|_{2}=\rho\}.
\]
 So, by a solution of \eqref{eq} we mean  a pair 
 $(\omega_{\rho}, u_\rho)\in\R\times H^{1}(\R^{d};\C)$ where $\omega_{\rho}$ appears as the Lagrange multiplier associated to the critical point $u_{\rho}$ on $S_{\rho}$.\medskip
For a fixed  mass $\rho>0$, we define the \emph{ground state energy} the quantity $I_{\rho^2}$, defined as
\begin{equation}\label{mini1}
I_{\rho^2}=\inf_{S_{\rho}} E(u).
\end{equation}

If a minimizer $u_\rho$ to \eqref{mini1} exists, we call it \emph{ground state solution}. 
Concerning the existence of ground states everything is well understood  in arbitrary dimension thanks to the following result:

 \begin{thm}\label{mainthm}[\cite{BFG2}, \cite{JL}]
Let $d\geq 1 $ and $1<q< p<1+\frac{4}{d}$, then there exists a strictly positive mass $\rho_0$ such that:\\
\textup{(i)} $ I_{\rho^2}=0$  for all  $\rho\in (0,\rho_0]$;\\
\textup{(ii)} $ I_{\rho^2}<0$  for all  $\rho\in (\rho_0, \infty)$.\smallskip

\noindent Moreover there are no constrained minimizers for $0<\rho<\rho_0$, and  for all  $\rho\in [\rho_0, \infty)$ there  exists $u_{\rho}\in S_{\rho}$ such that $I_{\rho^2}=E(u_{\rho})$. 
\end{thm}

In order to state our main results it is  convenient to introduce the notation
\begin{equation}\label{notation}
    \begin{aligned}
        A(u)&:=\frac12\int_{\R^d}|\nabla u(x)|^2dx,\\
B(u)&:=\frac{1}{q+1}\int_{\R^d}| u(x)|^{q+1}dx,\\
C(u)&:=\frac{1}{p+1}\int_{\R^d}|u(x)|^{p+1}dx\\
D(u)&:=\frac{1}{2}\int_{\R^d}| u(x)|^{2}dx.
    \end{aligned}
\end{equation}
such that 
 \[
 E(u)=A(u)+B(u)-C(u)
 \]
 and any standing wave solution $u$ solves
 \[
 \nabla E(u)+\omega \nabla D(u)=0.
 \]
The case $\omega=0$ corresponds to states that are time independent, i.e. static. Such kind of solution exists and it corresponds to the unique positive radially symmetric static solution $W$ that solves
\begin{equation}\label{eqstatic}
- \Delta W +|W|^{q-1} W-|W|^{p-1}W=0,
\end{equation}
whose asymptotic decay at infinity is explicitly given by a power-like tail that decays as $|x|^{\frac{2}{q-1}}$, see \cite{DS}. The assumed range of exponents ensures  that $W\in L^2(\R^d)$. \\
The classical Gagliardo-Nirenberg inequality with sharp constant $S$
\[
\|u\|_{q+1}\leq S \|\nabla u\|_{2}^{\theta}\|u\|_{p+1}^{1-\theta}
 \]
 where $\theta$ is defined by the relation
  \[
\frac{1}{q+1}=\frac{\theta}{2^{\star}}+\frac{1-\theta}{p+2},
 \]
 can be written in terms of $A(u), B(u),C(u)$ as  
 \begin{equation}\label{GN}
 C(u)\leq \frac{A(u)^{\alpha}B(u)^{\beta}}{\sigma}
 \end{equation}
 where
 \[
\sigma=\inf_{u\neq 0}\frac{A(u)^{\alpha}B(u)^{\beta}}{C(u)}
 \]
 and 
 \begin{equation}\label{relaz a-b}
 \alpha=\frac{2^{\star}}{2}\left(\frac{p-q}{2^{\star}-(q+1)}\right), \ \beta=\frac{2^{\star}-(p+1)}{2^{\star}-(q+1)},\ \sigma=\frac{p+1}{(q+1)^{\beta}}2^{-\alpha}S^{-(p+1)}.
 \end{equation}
 Note that the following (crucial) identity holds
\begin{equation}\label{relaz a e b}
\alpha+\beta=1+\frac{2\alpha}{d}.
\end{equation}
Now we introduce two relevant quantitites that are relevant for our first result.\\
We define 
\begin{equation}\label{def degli star}
\lambda_{\star}=\frac{2}{d}\beta^{-\frac{\beta}{\alpha+\beta-1}}\alpha^{\frac{\beta-1}{\alpha+\beta-1}}\sigma^{\frac{1}{\alpha+\beta-1}},\  \ C_{\star}=\frac{d}{2\alpha}\lambda_{\star}.
\end{equation}
such that
\begin{equation}\label{prop cstar}
C_{\star}^{\alpha+\beta-1}=\alpha^{-\alpha}\beta^{-\beta}\sigma.
\end{equation}
Our first result characterizes the energy level of the unique positive static solution that solves \eqref{eqstatic}.
\begin{thm}\label{thmfond}
Let $d\geq 3 $ and let us assume that $u\in \dot H^1(\R^d)\cap L^{q+1}(\R^d)$ fulfills
\[
E(u)\leq \lambda_{\star}, \ C(u)=C_{\star},
 \]
then $u$ is a critical point of $E$ at the level $\lambda_{\star}$ and it attains the sharp constant in GN inequality, i.e.
\begin{equation}\label{primoteo}
\left\{\begin{aligned}
&\nabla E(u)=0,\\
&E(u)=\lambda_{\star}, \\
&C(u)=C_{\star}=\frac{A(u)^{\alpha}B(u)^{\beta}}{\sigma}=\frac{A(u)}{\alpha}=\frac{B(u)}{\beta}.
\end{aligned}\right.
\end{equation}
Moreover, up to translations and multiplication by a constant phase, $u$ coincides with the unique positive radial solution $W$.
\end{thm}
Our second result introduces an alternative variational construction of critical points of the energy functional with a mass constraint. Here, instead to minimize the non convex functional $E$ on the $L^2$ sphere $\{D=\rho\}$, we choose to minimize the convex functional $D$ on a non convex constraint set  given by the conditions $E\leq \lambda, \ C\geq C_{\star}$, $\lambda<\lambda_{\star}$. Of course, for $u$ with $D(u)=\rho$ and $E(u)=\lambda$, minimizing $D$ on the set $\{E\leq \lambda, \ C\geq C_{\star}\}$is the same as minimizing $E$ on the set $\{D\leq \rho, \ C\geq C_{\star}\}$, both facts meaning   \[\{D\leq \rho,E\leq \lambda,   \ C\geq C_{\star}\}=\{D= \rho,E=\lambda,   \ C\geq C_{\star}\}. \]

As will become clear later, the characterization  of the set $E\leq \lambda, \ C\geq C_{\star}$ given by Theorem \ref{thmfond} will be crucial in the proof of the following theorem.
\begin{thm}\label{thmmain2}
Let $d\geq 3 $ and $1<q< p<1+\frac{4}{d}$. For every $\lambda \in (-\infty, \lambda_{\star})$, there exists $u\in H^1_{rad}(\R^d)$ that is a global minimizer of $D$ on the set
\[
\Omega_{\lambda}:=\{v\in H^1_{rad}(\R^d, \ E(v)\leq \lambda, \ C(v)\geq C_{\star}\},
 \]
that satisfies 
\begin{equation}\label{secondoteo}
\left\{\begin{aligned}
&\nabla E(u)+\omega \nabla D(u)=0,\\
& \omega>0,\\
&E(u)=\lambda, \\
&C(u)>C_{\star}.
\end{aligned}\right.
\end{equation}
\end{thm}

\begin{remark} The restriction to the radial space in Theorem~0.3 is used only to recover compactness. By restricting the minimization problem to $H^1_{\mathrm{rad}}(\mathbb R^d)$, we can use the compact embedding \[ H^1_{\mathrm{rad}}(\mathbb R^d) \hookrightarrow L^r(\mathbb R^d), \qquad 2<r<2^*. \] This yields the strong convergence of a minimizing sequence in $L^{q+1}(\mathbb R^d)$ and $L^{p+1}(\mathbb R^d)$, which is needed to pass to the limit in the nonlinear functionals $B$ and $C$. This compactness is crucial to ensure that the weak limit of a minimizing sequence remains in the constraint set. \end{remark}

\begin{remark} The variational characterization established in Theorem~\ref{thmmain2} appears to be particularly relevant in the context of nonlinear Schr\"odinger equations with competing nonlinearities. Indeed, the global minimizers for the mass on the set $\Omega(\lambda)$ constructed in Theorem~\ref{thmmain2}  correspond to standing waves  \[
 \psi(t,x)=e^{i\omega t}u(x)
 \]  for NLS equation \[ i\partial_t\psi+\Delta\psi = |\psi|^{q-1}\psi-|\psi|^{p-1}\psi, \qquad 1+\frac{2}{d}<q<p<1+\frac{4}{d}. \] 
Indeed Theorem ~\ref{thmmain2} shows that standing waves with mass strictly smaller than the threshold mass $\rho_0$ given by Theorem \ref{mainthm} exist. 
This fact is interesting is we connect our result with the paper by  Fukaya and Hayashi, see \cite{FH}, where they  proved that the  standing waves associated with the above equation can be dynamically unstable (under certain condition on $p$ and $q$), and where  they provide that the limiting profile of branches of positive-frequency ground states is the static solution $W$. Their analysis relies on a detailed variational description of the zero-frequency elliptic problem. Our results provide an alternative characterization of the same limiting profile through the threshold pair $(\lambda_*,C_*)$, showing that the static solution is the unique function satisfying the simultaneous constraints $E(u)\le \lambda_*$ and $C(u)=C_*$. This offers a new geometric interpretation of the distinguished role played by the algebraic ground state in the dynamics of the equation. \end{remark}
\begin{remark} Let $u_0$ be the ground state provided by Theorem~0.1, so that \[ E(u_0)=0, \qquad D(u_0)=\frac{\rho_0^2}{2}. \] We claim that \[ C(u_0)>C_*. \] Indeed, if $C(u_0)<C_*$, we can choose $t>1$ such that \[ C(tu_0)=C_*. \] Since $E(u_0)=0$, we have $ A(u_0)+B(u_0) = C(u_0), $ and therefore \[ E(tu_0) = (t^2-t^{p+1})\,A(u_0) + (t^{q+1}-t^{p+1})\,B(u_0). \] Because $p+1>q+1>2$ and $t>1$, both coefficients are negative, hence \[ E(tu_0)<0<\lambda_*. \] Since $C(tu_0)=C_*$, Theorem~0.2 would imply that $tu_0=W$ and therefore \[ E(tu_0)=E(W)=\lambda_*, \] a contradiction. Thus $C(u_0)\ge C_*$. The equality $C(u_0)=C_*$ is also impossible. Indeed, since $E(u_0)=0<\lambda_*$, Theorem~0.2 would again imply $u_0=W$, yielding \[ 0=E(u_0)=E(W)=\lambda_*, \] which is absurd. Consequently, \[ C(u_0)>C_*. \] It follows that $u_0\in\Omega(\lambda)$ for every $\lambda>0$. Hence for $0<\lambda<\lambda_*$,  the mass $\rho(\lambda)$ of the minimizer $u_\lambda$  fulfills \[ \rho(\lambda)\le \rho_0. \] Since every minimizer $u_\lambda$ of $\rho(\lambda)$ satisfies $E(u_\lambda)=\lambda$, the function $u_0$ cannot realize the minimum. Therefore, \[ \rho(\lambda)<\rho_0, \] for every $\lambda>0$, showing that the standing waves provided by Theorem~0.3 have mass strictly smaller than the threshold mass $\rho_0$ of Theorem~0.1. 
\end{remark}

\begin{remark}
 Since the seminal work by Tao, Vi\c san, and Zhang  \cite{TVZ}, the NLS with combined nonlinearities has attracted a lot of attention. Concerning  the existence of standing waves when the leading nonlinearity is mass super-critical it is worth mentioning at least the following recent works \cite{BFG23, SoaveJDE, SoaveJFA}, and references therein. 
\end{remark}

 \begin{remark} A further motivation for the present work comes from recent scattering results for nonlinear Schr\"odinger equations with combined mass-subcritical nonlinearities. In the recent paper \cite{BFG2} Bellazzini, Forcella and Georgiev proved a scattering result for datum belonging to $\Sigma$, i.e. the subspace of $H^1(\mathbb R^d)$ with finite variance. More  precisely  the authors proved that solutions scatter whenever the initial mass is sufficiently small, without requiring a smallness assumption on the whole $\Sigma$-norm. The sharp size of this threshold mass for scattering remains an open problem.  If $d\geq 3$ and taken $\lambda \in (-\infty, \lambda_{\star})$ and calling $u_{\lambda}$ the minimizer for the mass  given by Theorem~\ref{thmmain2} we proved that exist non scattering states with  mass strictly smaller then the  threshold mass $\rho_0$ given by Theorem \ref{mainthm}.
 \end{remark}

\section{Proof of Theorem \ref{thmfond}}
We recall an elementary inequality.
\begin{lem}\label{clemma}
For $x,y,\alpha,\beta\in \R^{+}$
\[
x^{\alpha}y^{\beta}\leq \frac{\alpha^{\alpha}\beta^{\beta}}{(\alpha+\beta)^{\alpha+\beta}}(x+y)^{\alpha+\beta},
 \]
 with equality iff $\frac{x}{\alpha}=\frac{y}{\beta}$.
\end{lem}
\begin{proof}
Calling $\tilde x=\frac{x}{x+y}$, $\tilde y=\frac{y}{x+y}$, $\tilde \alpha=\frac{\alpha}{\alpha+\beta}$, $\tilde \beta=\frac{\beta}{\alpha+\beta}$, we have
$\tilde x+\tilde y=1$, $\tilde \alpha+\tilde \beta=1$.
The concave function $f(\tilde x)=\tilde x^{\tilde \alpha}(1-\tilde x)^{1-\tilde \alpha}$ attains the maximum in $[0,1]$ at $\tilde x=\tilde \alpha$ such that
$f(\tilde x)\leq {\tilde \alpha}^{\tilde \alpha }(1-\tilde \alpha)^{1-\tilde \alpha }$. Hence we have
\[
(\frac{x}{x+y})^\frac{\alpha}{\alpha+\beta}(\frac{y}{x+y})^\frac{\beta}{\alpha+\beta}\leq (\frac{\alpha}{\alpha+\beta})^\frac{\alpha}{\alpha+\beta}(\frac{\beta}{\alpha+\beta})^\frac{\beta}{\alpha+\beta},
 \]
 which implies 
 \[
x^{\alpha}y^{\beta}\leq \frac{\alpha^{\alpha}\beta^{\beta}}{(\alpha+\beta)^{\alpha+\beta}}(x+y)^{\alpha+\beta},
 \]
 with equality iff $\frac{x}{\alpha}=\frac{y}{\beta}$.
\end{proof}
Now let us consider $u\in \dot H^1(\R^d)\cap L^{q+1}(\R^d)$ fulfilling
\[
E(u)\leq \lambda_{\star}, \ C(u)=C_{\star}.
 \]
Then, by GN inequality \eqref{GN} and Lemma \ref{clemma}, hiding the dependence on $u$, we obtain
\[
C_{\star}=C\leq \sigma^{-1} A^{\alpha}B^{\beta}\leq  \sigma^{-1}\frac{\alpha^{\alpha}\beta^{\beta}}{(\alpha+\beta)^{\alpha+\beta}}(A+B)^{\alpha+\beta}=
 \]
 \[
=\sigma^{-1}\frac{\alpha^{\alpha}\beta^{\beta}}{(\alpha+\beta)^{\alpha+\beta}}(C+E)^{\alpha+\beta}\leq \sigma^{-1}\frac{\alpha^{\alpha}\beta^{\beta}}{(\alpha+\beta)^{\alpha+\beta}}(C_{\star}+\frac{2\alpha}{d}C_{\star})^{\alpha+\beta},
 \]
 where we have used in the last line the relation $\lambda_{\star}=\frac{2\alpha}{d}C_{\star}$.

Hence  we obtain
  \[
C_{\star}\leq \sigma^{-1}\ \frac{\alpha^{\alpha}\beta^{\beta}}{(\alpha+\beta)^{\alpha+\beta}}(C_{\star}+\frac{2\alpha}{d}C_{\star})^{\alpha+\beta}= \sigma^{-1}\frac{\alpha^{\alpha}\beta^{\beta}}{(\alpha+\beta)^{\alpha+\beta}}(C_{\star}(1+\frac{2\alpha}{d}))^{\alpha+\beta}=
 \]
\[
=\sigma^{-1}\frac{\alpha^{\alpha}\beta^{\beta}}{(\alpha+\beta)^{\alpha+\beta}}(C_{\star}(\alpha+\beta))^{\alpha+\beta}=C_{\star}.
 \]
 thanks to the relation \eqref{relaz a e b} and \eqref{prop cstar}. Therefore the above are in fact equalities and in particular
 \[ 
 A^{\alpha}B^{\beta}= \frac{\alpha^{\alpha}\beta^{\beta}}{(\alpha+\beta)^{\alpha+\beta}}(A+B)^{\alpha+\beta}
 \]
 and thus $\frac{A}{\alpha}=\frac{B}{\beta}$. Calling $\frac{A}{\alpha}=\frac{B}{\beta}=c$ we get $A+B=(\alpha+\beta)c$  such that
\[ 
 C_{\star}=C=\sigma^{-1} \alpha^{\alpha}\beta^{\beta}c^{\alpha+\beta}
 \]
 
that implies, thanks to \eqref{prop cstar}, $c=C=\frac{A}{\alpha}=\frac{B}{\beta}$. Finally, thanks to the fact that $C=\sigma^{-1}A^{\alpha}B^{\beta}$, $u$ attains the sharp constant
in Gagliardo-Nirenberg inequality and minimizes the ratio $\frac{A^{\alpha}B^{\beta}}{C}$. Thus
\[
0=\nabla (\ln \frac{A^{\alpha}B^{\beta}}{C})=\frac{\alpha}{A}\nabla A+\frac{\beta}{B}\nabla B-\frac{1}{C}\nabla C=\frac{1}{C}\left(\nabla E\right).
\]
Now to conclude that $u=W$ (modulo space and phase shift) it sufficies to notice that 
\begin{itemize}
\item any solution $u\in \dot H^1(\R^d)\cap L^{q+1}(\R^d)$ to the equation $\nabla E(u)=0$ shall fulfill the relations
 \begin{equation}
\left\{\begin{aligned}
&2A(u)+(q+1)B(u)-(p+1)C(u)=0,\\
&(d-2)A(u)+dB(u)-dC(u)=0,
\end{aligned}\right.
\end{equation}
that are equivalent to
 \begin{equation}
\left\{\begin{aligned}
&A(u)=\alpha C(u),\\
&B(u)=\beta C(u),\\
&C(u)=\frac{d}{2\alpha}E(u).
\end{aligned}\right.
\end{equation}
\item the GN inequality implies
\[
\frac{A(u)^{\alpha}B(u)^{\beta}}{C(u)}=(\frac{d}{2})^{\alpha+\beta-1}\beta^{\beta}\alpha^{1-\beta}E(u)^{\alpha+\beta-1}\geq \sigma,
\]
and hence an a priori bound of energy of non trivial critical point of $E$ is given by $E\geq \lambda_{\star}.$
\item the optimizers for GN can be obtained also by minimizing (a la Berestycki-Lions \cite{BL}) $A(u)$ on the constraint $B(u)-C(u)=1$ and then rescaling to remove the Lagrange multiplier.
\item any minimizer for  $A(u)$ on the constraint $B(u)-C(u)=1$ is radially symmetric (Lopes \cite{L}) and it is positive ( Berestycki-Lions \cite{BL}) hence it is unique (Kwong-Zhang \cite{KZ}).
\end{itemize}
\section{Proof of Theorem \ref{thmmain2}}
 Notice that $\Omega(\lambda)\neq\emptyset$ for every $\lambda\in\mathbb R$. Indeed, for any $v\in H^1(\mathbb R^d)\setminus\{0\}$, \[ E(tv) = t^2A(v) +t^{q+1}B(v) -t^{p+1}C(v) \to -\infty \] as $t\to\infty$, while $C(tv)\to+\infty$. Therefore $tv\in\Omega(\lambda)$ for $t$ large enough. Now let us consider a minimizing sequence $u_{j}$ for $D$ on $\Omega_{\lambda}$ and denote $A_j=A(u_j)$,  $B_j=B(u_j)$, $C_j=C(u_j)$, $D_j=D(u_j)$, thus
\[
D_j \rightarrow d_{\lambda}:=\inf_{\Omega_{\lambda}}D.
\]
By the fact that $u_j \in \Omega_{\lambda}$ we have $A_j<C_j+\lambda$ and hence by GN inequality
\[
C_j \leq s_1A_j^{\frac{d(p-1)}{4}}D_{j}^{\frac{2(p+1)-d(p-1)}{4}}=O(C_j^{\frac{d(p-1)}{4}}).
\]
The mass subcritical condition $p<1+\frac{4}{d}$ guarantees hence that $C_j$ is bounded as well as  $A_j$ and that $u_j\rightharpoonup u$ in $H^1_{rad}(\R^d)$.
Denoting   $A=A(u)$,  $B=B(u)$, $C=C(u)$, $D=D(u)$ we have, by the compact embedding of radial function, 
 \begin{equation}
\left\{\begin{aligned}
&A_j\rightarrow a\geq A,\\
&B_j\rightarrow B,\\
&C_j \rightarrow C\\
&D_j \rightarrow d_{\lambda}\geq D.
\end{aligned}\right.
\end{equation}
Thus $u \in \Omega_{\lambda}$ and $D\leq \inf_{\Omega_{\lambda}}D$, so $D=\delta_{\lambda}$ and the infimum is achieved. We have also that  $C(u)>C_{\star}$
thanks to Theorem \ref{thmfond} because $\lambda<\lambda_{\star}$. Moreover if we assume $E(u)<\lambda$, $u$ is forced to be  a free critical point of $D$ which is the trivial function, 
therefore $E(u)=\lambda$.
Moreover $\nabla E(u)\neq 0$ by Theorem \ref{thmfond} and by Lagrange multiplier principle we have for some $\omega\neq 0$
\[
\nabla E+\omega \nabla D=0.
\]
Now, in order to prove that $\omega>0$, we argue by contradiction  assuming  that $\omega<0$. For $0<t<1$ certainly $tu \notin \Omega_{\lambda}$ because  $D(tu)=t^2D(u)<D(u)=\inf_{\Omega_{\lambda}}D $, and since  $C(tu)>C_*$ for  $t$ close to $1$ one has $E(tu) >\lambda$ for all $t<1$ close to $1$. Thus $\frac{d}{dt}E(tu)\big|_{t=1}=(\nabla E(u)\cdot u)\leq 0$, whence from $(\nabla E(u)\cdot u)+\omega (\nabla D(u)\cdot u)=0$ we have $\frac{d}{dt}E(tu)\big|_{t=1}=-\omega (\nabla D(u)\cdot u)>0.$

\end{document}